\documentclass[12pt]{article}%
   
\usepackage{amsmath,enumerate}
\usepackage{amsfonts}
\usepackage{amssymb}

\newtheorem{theorem}{Theorem}[section]

\newtheorem{definition}[theorem]{Definition}

\newtheorem{lemma}[theorem]{Lemma}

\newenvironment{proof}[1][Proof]{\noindent\textbf{#1.} }
{\hfill \ \rule{0.5em}{0.5em}}

\newcommand{\points}{\mathcal{P}}
\newcommand{\lines}{\mathcal{L}}
\newcommand{\incidence}{\mathcal{I}}

\newcommand{\fq}{\mathbb{F}_q}
\newcommand{\intq}{\mathbb{Z}_{q^2-1} }

\newcommand\blfootnote[1]{%
  \begingroup
  \renewcommand\thefootnote{}\footnote{#1}%
  \endgroup
}

\begin{document}

\title{\vspace{-1cm}Sidon sets and graphs without 4-cycles}
\author{Michael Tait\thanks{Department of
Mathematics, University of California, San Diego, CA, USA, mtait@math.ucsd.edu} \and Craig Timmons\thanks{Department of
Mathematics, University of California, San Diego, CA, USA, ctimmons@math.ucsd.edu} }

\maketitle
\blfootnote{\textsuperscript{1}Both authors were partially supported by NSF Grant DMS-1101489 through Jacques Verstra\" ete }

\vspace{-1cm}

\begin{abstract}
The problem of determining the maximum number of edges in an $n$-vertex graph that does not contain a 4-cycle has a rich history in extremal graph theory.  Using Sidon sets constructed by Bose and Chowla, for each odd prime power $q$ we construct a graph with $q^2 - q - 2$ vertices that does not contain a 4-cycle and has at least $\frac{1}{2}q^3 - q^2 - O(q^{3/4})$ edges.  This disproves a conjecture of Abreu, Balbuena, and Labbate concerning the Tur\'{a}n number $\textup{ex}(q^2 - q - 2 , C_4)$.    
\end{abstract}

\section{Introduction}

Let $F$ be a graph.  The \emph{Tur\'{a}n number} of $F$, denoted $\textup{ex}(n, F)$, is the maximum number of edges in an $n$-vertex graph that does not contain $F$ as a subgraph.  Determining $\textup{ex}(n , F)$ for different graphs $F$ is one of the central problems in extremal combinatorics.  One of the most studied cases is the Tur\'{a}n number of $C_4$, the cycle on four vertices.  It is known that $\textup{ex}(n , C_4) \leq \frac{1}{2} n^{3/2} + o(n^{3/2})$ for every $n \geq 1$ (see \cite{bol}).  It is more difficult to construct $n$-vertex graphs without 4-cycles that have $\frac{1}{2}n^{3/2} + o(n^{3/2})$ edges.  Using polarity graphs of projective planes Brown \cite{b}, Erd\H{o}s, R\'{e}nyi, and S\'{o}s \cite{ers} independently proved that for each prime power $q$, $\textup{ex}( q^2  + q + 1 , C_4) \geq \frac{1}{2}q (q+1)^2$.  To define polarity graphs we need some terminology from finite geometry.  

Let $\points$ and $\lines$ be disjoint sets and $\incidence \subset \points \times \lines$.  Elements of $\points$ are called \emph{points}, elements of $\lines$ are called \emph{lines}, and $\incidence$ defines an incidence relation on 
the pair $( \points , \lines)$.  Let $\pi : \points \cup \lines \rightarrow \points \cup \lines$ be a bijection such that 
$\pi ( \points ) = \lines$, $\pi ( \lines ) = \points$, $\pi^2 = \textup{id}$, and for all 
$p \in \points$ and $l \in \lines$ we have $(p , l) \in \incidence$ if and only if $( \pi (l) , \pi (p ) ) \in \incidence$.  The map $\pi$ is a \emph{polarity} of the geometry $( \points , \lines , \incidence )$.  The \emph{polarity graph} $G_{\pi}$ of 
the geometry $(\points , \lines , \incidence )$ with respect to $\pi$ is the graph with vertex set $V(G_{\pi}) = \points$ and 
edge set 
\[
E(G_{\pi}) = \{  \{p,q \} : p , q \in \points, p \neq q, \mbox{and}~ (p , \pi (q) ) \in \incidence \}.
\]
A point $p$ is an \emph{absolute point} of $\pi$ if $( p , \pi (p) ) \in \incidence$.  

If $( \points , \lines , \incidence)$ is a finite projective plane of order $q$ and $\pi$ is an orthogonal polarity (one with exactly $q+1$ absolute points) then the polarity graph will have $q^2 + q  +1 $ vertices, $\frac{1}{2}q (q+1)^2$ edges, and will not contain a 4-cycle.  The constructions of \cite{b} and \cite{ers} are polarity graphs of the projective plane $PG(2 , \mathbb{F}_q )$ where $q$ is a prime power and $\mathbb{F}_q$ is the finite field with $q$ elements.  The polarity is the orthogonal polarity sending the point $(x_0 , x_1 , x_2)$ to the line $[x_0 , x_1 , x_2]$ and vice versa (see \cite{bol} or \cite{fs} for more details).  These polarity graphs show that 
for any prime power $q$, $\textup{ex}( q^2 +q + 1 , C_4) \geq \frac{1}{2}q ( q + 1)^2$.      

The exact value of $\textup{ex}(n, C_4)$ was determined using computer searches 
(\cite{c}, \cite{y}) for all $n \leq 31$.    
F\"{u}redi \cite{fu} proved that whenever $q \geq 13$ is a prime power, 
$\textup{ex}(q^2 + q + 1 , C_4) \leq \frac{1}{2} q( q+1)^2$ thus we get the exact result $\textup{ex}(q^2 + q + 1, C_4) = \frac{1}{2}q(q+1)^2$ for all prime powers $q \geq 13$.  It was also shown in \cite{fu} that the only graphs with $q^2  +q + 1$ vertices and $\frac{1}{2} q ( q + 1)^2$ edges that do not contain 4-cycles are orthogonal polarity graphs of finite projective planes.  Along with the constructions of \cite{b} and \cite{ers}, the results of F\"{u}redi are the most important contributions to the 4-cycle Tur\'{a}n problem.  Recently Firke, Kosek, Nash, and Williford \cite{fknw} proved that for even $q$, 
$\textup{ex}(q^2  + q , C_4) \leq \frac{1}{2}q (q+1)^2 - q$.  If $q$ is a power of two then we have the exact result 
$\textup{ex}(q^2 + q , C_4) = \frac{1}{2}q (q+1)^2 - q$.  The lower bound in this case comes from taking an orthogonal polarity graph of a projective plane of order $q$ and removing a vertex of degree $q$.  

The results we have mentioned so far describe all of the cases in which an exact formula for $\textup{ex}(n, C_4)$ is known.  Using known results on densities of primes one has the asymptotic result $\textup{ex}(n,C_4)  = \frac{1}{2}n^{3/2} + o(n^{3/2})$ but there are still many open problems concerning graphs with 4-cycles.  For example, Erd\H{o}s and Simonovits \cite{es} conjectured that if $G$ is any $n$-vertex graph with 
$\textup{ex}(n , C_4) + 1$ edges then $G$ must contain at least $n^{1/2} + o( n^{1/2})$ copies of $C_4$.  
For more on the Tur\'{a}n problem for $C_4$ and other bipartite Tur\'{a}n problem we refer the reader to the excellent survey 
of F\"{u}redi and Simonovits \cite{fs}.   

While investigating adjacency matrices of polarity graphs, Abreu, Balbuena, and Labbate \cite{abl} were able to find subgraphs of a polarity graph that have many edges.  By deleting such a subgraph, Abreu et.\ al.\ \cite{abl} proved that for any prime power $q$, 
\[
\textup{ex}( q^2 - q - 2 , C_4) \geq 
\left\{
\begin{array}{ll}
\frac{1}{2}q^3 - q^2 - \frac{q}{2} + 1 & \mbox{if $q$ is odd,} \\ 
\frac{1}{2}q^3 - q^2 & \mbox{if $q$ is even}.
\end{array}
\right.
\]
They conjectured that these bounds are best possible.  Our main result shows that when $q$ is an odd prime power, this lower bound can be improved by $\frac{q}{2} - O(q^{3/4})$.        

\begin{theorem}\label{main}
If $q$ is an odd prime power then 
\[
\textup{ex}( q^2 - q - 2 , C_4) \geq \frac{1}{2}q^3 - q^2 - O ( q^{3/4} ).
\]
\end{theorem}

We will construct graphs without 4-cycles using the Sidon sets constructed by Bose and Chowla \cite{bc}.  Let $\Gamma$ be an abelian group.  A set $A \subset \Gamma$ is a \emph{Sidon set} if whenever $a+b = c+d$ with $a,b,c,d \in A$, the pair $(a,b)$ is a permutation of $(c,d)$.  Sidon sets are well studied objects in combinatorial number theory and for more on Sidon sets we recommend O'Bryant's survey \cite{o}.   

Let $q$ be a prime power and $\theta$ be a generator of the multiplicative group $\mathbb{F}_{q^2}^*$ where 
$\mathbb{F}_{q^2}^*$ is the nonzero elements of the finite field $\mathbb{F}_{q^2}$.  
Bose and Chowla proved \cite{bc} that  
\[
A(q , \theta ) := \{ a \in \mathbb{Z}_{q^2 - 1} : \theta^a - \theta \in \fq \}
\]
is a Sidon set in the group $\mathbb{Z}_{q^2 - 1}$.  

\begin{definition}\label{def}
Let $q$ be a prime power and $\theta$ be a generator of the multiplicative group $\mathbb{F}_{q^2}^*$.  The graph $G_{q, \theta }$ is the graph with vertex set $\mathbb{Z}_{q^2 - 1}$ and two distinct vertices $i$ and $j$ are adjacent if and only if $i + j  = a$ for some 
$a \in A(q, \theta )$.  
\end{definition}

It is known that Sidon sets can be used to construct graphs without 4-cycles.  We will prove a result about the Bose-Chowla Sidon sets (see Lemma~\ref{bclemma}) that helps us find a subgraph of $G_{q, \theta}$ with $q+1$ vertices that contains many edges.  We remove this subgraph to obtain a graph with $q^2  - q - 2$ vertices and at least 
$\frac{1}{2} q^3 - q^2 - O( q^{3/4} )$ edges.  In addition to providing examples of graphs with no 4-cycles, the graphs $G_{q, \theta}$ have been used to solve other extremal problems (see \cite{fan}).  
 
We would like to remark that we could have defined $G_{q, \theta}$ as a polarity graph in the following way.  Let $ \mathcal{P} = \mathbb{Z}_{q^2 - 1}$ and let $\mathcal{L}$ be the set of $q^2  -1$ translates of $A(q , \theta)$.  That is, 
$\mathcal{L} = \{ A_1 , A_2 , \dots  , A_{q^2 - 1} \}$ where $A_i := A(q , \theta ) + i$.  This defines a geometry in the obvious way; $i \in \points$ is incident to $A_j \in \lines$ if and only if $i \in A_j$.  We define a polarity by $\pi (i ) = A_{q^2  - 1 - i}$ for all $i \in \mathcal{P}$, and $\pi ( A_i) = q^2 - 1 - i$ for all $A_i \in \mathcal{L}$.  The fact that $\pi$ is a polarity can be checked directly.  We choose to use Definition~\ref{def} as it is more convenient for our argument.


\section{Proof of Theorem~\ref{main}}
 
In this section we fix an odd prime power $q$ and a generator $\theta$ of the multiplicative group $\mathbb{F}_{q^2}^*$.  
We write $A$ for the Sidon set $A( q, \theta )$ in $\mathbb{Z}_{q^2 - 1}$ and observe that $|A|  = q$.  All of our manipulations will be done in the group $\mathbb{Z}_{q^2 - 1}$ or in the finite field $\mathbb{F}_{q^2}$.  If it is not clear from the context we will state which algebraic structure we are working in.      

The first two lemmas are known.  We present proofs for completeness.  

\begin{lemma}\label{lemma1}
The graph $G_{q , \theta}$ does not contain a 4-cycle.
\end{lemma}
\begin{proof}
Suppose $ijkl$ is a 4-cycle in $G_{q , \theta}$.  There are elements $a,b,c,d \in A$ such that 
$i+ j = a$, $j + k = b$, $k + l = c$, and $l + i = d$.  This implies 
\[
a + c = b + d.
\]
Since $A$ is a Sidon set, $(a,c)$ is a permutation of $(b,d)$.  If $a = b$ then $i + j = j + k$ so $i = k$.  
If $a = d$ then $i + j = l + i$ so $j=l$.  In either case we have a contradiction thus $G_{q , \theta}$ does not contain a 4-cycle.   
\end{proof}

\begin{lemma}\label{lemma4}
If $A - A := \{ a - b : a , b \in A \}$ then  
\[
A - A = \intq \backslash \{ q + 1 , 2(q+1) , 3(q+1) , \dots , (q -2)( q + 1) \}.
\]
\end{lemma}
\begin{proof}
Suppose $s(q + 1) \in A - A$ for some $1 \leq s \leq q -2$.  Write $s(q+1) = a - b$ where $a,b \in A$ and $a \neq b$.  
We have for some $\alpha , \beta \in \mathbb{F}_q$, 
\[
\theta^{ s ( q + 1) } = \theta^{ a - b} = \theta^a \theta^{-b} = ( \theta + \alpha ) ( \theta + \beta)^{-1}.
\]
From this we obtain 
\[
\theta + \alpha = ( \theta + \beta )( \theta^{q+1})^s
\]
but $\theta^{q+1} \in \mathbb{F}_q$ so $\theta + \alpha = ( \theta + \beta ) \gamma$ for some $\gamma \in 
\mathbb{F}_q$.  Since $\theta$ does not satisfy a nontrivial linear relation over $\mathbb{F}_q$ we must have 
$\gamma = 1$ hence $\alpha = \beta$ (in $\mathbb{F}_{q^2}$) so $a = b$ (in $\intq$).  From this we get $s(q+1) = 0$ which contradicts the fact that $1 \leq s \leq q - 2$.  This shows that 
\[
(A - A) \cap \{ q + 1 , 2(q + 1) , \dots , ( q- 2)(q + 1) \} = \emptyset.
\]
Since $A$ is a Sidon set, $|A - A| = q(q-1) + 1$ which is precisely the number of elements in the set
\[
\mathbb{Z}_{q^2 - 1} \backslash \{ q+1 , 2(q+1 ) , \dots , (q - 2) ( q +1) \}
\]
and this completes the proof of the lemma.    
\end{proof}

\bigskip

Let $i$ be a vertex in $G_{q , \theta}$.  If $i + i \in A$ then the degree of $i$ is $q-1$.  If $i + i \notin A$ then the degree of $i$ is $q$.  We call a vertex of degree $q-1$ an \emph{absolute point}.  

\begin{lemma}\label{lemma5}
Distinct vertices $i$ and $j$ in $G_{q , \theta}$ have a common neighbor if and only if 
$i  - j \in (A - A) \backslash \{ 0 \}$.
\end{lemma}
\begin{proof}
First suppose $i$ and $j$ are distinct vertices that have a common neighbor $k$.  Then $i + k = a$ and $k + j =b$ for some $a,b \in A$ so $i - j = ( a - k ) - (b - k ) = a - b$.  Since $i \neq j$, we get that $a -b \neq 0$.  

Now suppose $i - j = a - b$ for some $a,b \in A$ with $a \neq b$.  Let $k = a - i$.  Then $k + i = a$ so $k$ is adjacent to $i$.  Also, 
$k = a - i = b - j $ so $k + j = b$ and $k$ is adjacent to $j$.  
\end{proof}

\begin{lemma}\label{lemma2}
If $i$ is an absolute point then $i + \frac{q^2 - 1}{2}$ is also an absolute point.
\end{lemma}
\begin{proof}
If $2i = a$ for some $a \in A$ then $2(i+ \frac{q^2-1}{2} ) = 2i = a$.  
\end{proof}

\begin{lemma}\label{lemma6}
Let $i$ and $j$ be two distinct absolute points of $G_{q, \theta }$.  
If $i \neq j + \frac{q^2 - 1}{2}$ then $i$ and $j$ have a common neighbor and if $i = j+ \frac{q^2 - 1}{2}$ then $i$ and $j$ do not have a common neighbor.  
\end{lemma}
\begin{proof}
By Lemmas~\ref{lemma4} and~\ref{lemma5}, $i$ and $j$ have a common neighbor unless $ i   -  j = s(q + 1)$ for some 
$1 \leq s \leq q -2$.  Since $i$ and $j$ are absolute points, there exists elements $a,b \in A$ such that 
$2i  = a$ and $2j = b$ thus $a - b = 2s (q  + 1)$.  By Lemma~\ref{lemma4}, it must be the case that $a = b$ so 
$2i = 2j$.  The solutions to $2x \equiv 2y ( \textup{mod}~q^2 - 1)$ are $x = y$ and $x = y + \frac{q^2  - 1}{2}$ hence 
$i = j$ or $i = j + \frac{q^2 - 1}{2}$.  Thus $i$ and $j$ will have a common neighbor whenever they are distinct absolute points with $i \neq j + \frac{q^2 - 1}{2}$ and will not have a common neighbor when $i = j + \frac{q^2 - 1}{2}$.    
\end{proof}

\begin{lemma}\label{bclemma}
Let $\{a_1,a_2,a_3 \}$ and $\{b_1,b_2 , b_3 \}$ be subsets of $A$ with $a_1,a_2$, and $a_3$ all distinct and $b_1, b_2$, and $b_3$ all distinct.  If  
\[
2b_1 - a_1 = 2b_2 - a_2 = 2b_3 - a_3
\]
then two of the ordered pairs $(a_1,b_1)$, $(a_2, b_2)$, $(a_3, b_3)$ are equal.  
\end{lemma}

The proof of Lemma~\ref{bclemma} is simple but it is not short.  For this reason we postpone the proof until after the proof of Theorem~\ref{main}.   

\begin{lemma}\label{at most 2}
Any vertex $j$ is adjacent to at most two absolute points.
\end{lemma}
\begin{proof}
Suppose $j$ is a vertex of $G_{q, \theta}$ that is adjacent to three distinct absolute points $i_i, i_2$, and $i_3$.  
There exists elements $a_1,a_2,a_3,b_1,b_2,b_3 \in A$ such that 
\[
2i_k = a_k ~~ \mbox{and} ~~ i_k + j = b_k
\]
for $k=1,2,3$.  Since $i_1,i_2,i_3$ are all distinct, $b_1, b_2$, and $b_3$ must all be distinct.  If 
$a_k = a_l$ for some $1 \leq k < l \leq 3$ then $i_k = i_{l} + \frac{q^2 - 1}{2}$.  In this case, the vertices $i_k$ and $i_{l}$ are absolute points with a common neighbor but this is impossible by Lemma~\ref{lemma6}.  We conclude that 
$a_1, a_2$, and $a_3$ are all distinct.  For each $k$, we can write $i_k + j = b_k$ as 
$2j  = 2b_k - a_k$ so that 
\[
2b_1 - a_2 = 2b_2 - a_2 = 2b_3 - a_3.  
\]
By Lemma~\ref{bclemma}, $(a_k  , b_k) = (a_l , b_l)$ for some 
$1 \leq k < l \leq 3$ but we have already argued that $a_k$ and $a_l$ are distinct.  This gives the needed contradiction and completes the proof of the lemma.    
\end{proof}

\bigskip

\begin{proof}[Proof of Theorem~\ref{main}]
Let $P$ be the absolute points of $G_{q , \theta}$.  By 
Lemma~\ref{lemma2}, the absolute points come in pairs so we can write
\[
P = \{ i_1 , i_1 + \frac{q^2 - 1}{2} , i_2 , i_2 + \frac{q^2 - 1}{2} , \dots , i_t , i_t + \frac{q^2 - 1}{2} \}
\]
where $2t$ is the number of absolute points of $G$.  When $q$ is odd, $q^2 - 1$ is even and we can write 
$q^2 - 1 = 2^r m$ where $r \geq 1$ is an integer and $m$ is odd.  If $a \in A$ then the congruence 
\[
2x \equiv a ( \textup{mod}~2^rm )
\]
has no solution when $a$ is odd and two solutions if $a$ is even.  Therefore $t$ is exactly the number of even elements of $A$ when we view $A$ as a subset of $\mathbb{Z}$.  Lindstr\"{o}m \cite{l2} proved that dense Sidon sets are close to evenly distributed among residue classes.  In particular, the results of \cite{l2} imply that 
\begin{equation}\label{eqt}
t = \frac{q}{2} + O( q^{3/4})
\end{equation}
so we know that we have $q + O (q^{3/4})$ absolute points in $G_{q,\theta}$.    
The number of vertices of $G_{q , \theta}$ is $q^2 - 1$ and the number of edges of $G_{q , \theta}$ is 
\begin{eqnarray*}
e(G)  =  \frac{1}{2} \left(  q(q^2 - 1 - 2t) + ( q-1)(2t) \right) = \frac{1}{2}q^3 - \frac{1}{2}q - t.
\end{eqnarray*}

Let $S \subset V(G_{q, \theta})$ with $|S| = q+1$ and let $t_S$ be the number of absolute points in $S$.  The graph 
$G_{q, \theta} \backslash S$ has $q^2 - q - 2$ vertices and 
\begin{equation}\label{equation 1}
\frac{1}{2}q^3 - \frac{1}{2} q - t - e(S) - e(S , \overline{S} )
\end{equation}
edges.  Here $e(S , \overline{S})$ is the number of edges of $G_{q , \theta}$ with exactly one endpoint in $S$.    
We can rewrite $e(S) + e(S , \overline{S})$ as 
\[
e(S) + e(S , \overline{S} ) = \sum_{i \in S} d(i) - e(S) = ( q + 1 - t_S ) q + t_S (q - 1) - e(S) = q^2 + q - t_S - e(S).
\]
By (\ref{equation 1}) we can write the number of edges of $G_{q , \theta} \backslash S$ as 
\begin{equation}\label{equation 2}
\frac{1}{2}q^3 - \frac{1}{2} q - t - ( q^2 + q - t_S - e(S) ) = \frac{1}{2}q^3 - q^2 - \frac{3}{2}q - t + t_S +  e(S).
\end{equation}

For any $1  \leq j_1 < j_2 \leq t$, the pair $i_{j_1}$ and $i_{j_2}$ of absolute points have a unique common neighbor by Lemmas~\ref{lemma6} and~\ref{lemma1}.  
Set $k = \lfloor \frac{1}{2} \sqrt{ 8 q + 9 }  - \frac{1}{2} \rfloor$ and note that for large enough $q$ we have $k \leq t$.  The integer $k$ is chosen so that it is as large as possible and still satisfies the inequality $\binom{k}{2} + k \leq q +1$.  
Let $S_1 = \{ i_1 , \dots , i_k \}$.  For each pair $1 \leq j_1 < j_2 \leq k$, let $x_{j_1 , j_2}$ be the unique common neighbor of the absolute points $i_{j_1}$ and $i_{j_2}$.  Let 
$S_2 = \{ x_{j_1 , j_2} : 1 \leq j_1 < j_2 \leq k \}$.  By Lemma~\ref{at most 2}, $S_2$ consists of $\binom{k}{2}$ distinct vertices.
A short calculation shows that 
$\binom{k}{2} + k \geq q - O ( \sqrt{q} )$.  Let $S_3$ be a set of $q+1 - \binom{k}{2} - k$ vertices chosen arbitrarily from 
$V( G_{q,\theta} ) \backslash ( S_1 \cup S_2)$.  Let $S$ be the subgraph of $G_{q, \theta}$ induced by the vertices
$S_1 \cup S_2 \cup S_3$.  By construction, $S$ has $q+1$ vertices and 
at least $2 \binom{k}{2}$ edges so 
\[
t_S + e(S) \geq k + 2 \binom{k}{2} \geq 2q  - O ( \sqrt{q} ).  
\]
By (\ref{eqt}) and (\ref{equation 2}), removing the vertices of $S$ from $G_{q , \theta}$ leaves a graph with $q^2 - q - 2$ vertices and at least 
\[
\frac{1}{2} q^3 - q^2 - 2q  + 2q - O ( q^{3/4} ) = \frac{1}{2}q^3 - q^2 - O(q^{3/4}).
\]
edges.  
\end{proof}  

\bigskip

Now we return to the proof of Lemma~\ref{bclemma}.  

\begin{proof}[Proof of Lemma~\ref{bclemma}]
Let $\{a_1,a_2,a_3 \} , \{b_1,b_2,b_3 \} \subset A$ with $a_1,a_2$, and $a_3$ all distinct, and $b_1,b_2$, and $b_3$ all distinct.  
Since $a_k,b_k \in A$, there exists elements $c_k , d_k \in \mathbb{F}_q$ such that 
\[
\theta^{a_k}  = \theta + c_k ~~ \mbox{and} ~~ \theta^{b_k} = \theta + d_k
\]
for $k=1,2,3$.  Observe that $c_1,c_2$, and $c_3$ are all distinct and so are $d_1,d_2$, and $d_3$. 

The generator $\theta$ satisfies a degree two polynomial over $\mathbb{F}_q$, say $\theta^2 = \alpha \theta + \beta$ where $\alpha , \beta \in \mathbb{F}_q$.  Since $\theta$ generates $\mathbb{F}_{q^2}^{*}$, it cannot be the case that $\alpha = 0$ and if $\beta  = 0$, then $\theta ( \theta - \alpha ) = 0$ which is impossible since $\theta \notin \mathbb{F}_q$.  The polynomial 
$X^2 - 3X + 3 \beta \in \mathbb{F}_q [X]$ has at most two roots in $\mathbb{F}_q$.  Without loss of generality, we may assume that 
\begin{equation}\label{equ 0}
c_1^2 - 3c_1 \alpha + 3 \beta \neq 0
\end{equation}
since $c_1,c_2$, and $c_3$ are all distinct.  This fact will be important towards the end of the proof.  

Consider the equation $2b_1 + a_2 = 2b_2 + a_1$.  We can rewrite this as 
\[
( \theta + d_1)^2 ( \theta + c_2) = ( \theta + d_2 )^2 ( \theta  + c_1).
\]
If we expand, use $\theta^2 = \alpha \theta + \beta$, and regroup we obtain 
\begin{eqnarray*}
~&~& \theta ( 2d_1 \alpha + c_2 \alpha + d_1^2 + 2d_1 c_2 )  + ( 2d_1 \beta + c_2 \beta + d_1^2 c_2) \\
~&=& \theta ( 2d_2 \alpha + c_1 \alpha  + d_2^2 + 2d_2 c_1) + ( 2d_2 \beta + c_1 \beta + d_2^2 c_1).
\end{eqnarray*}
These coefficients are all in $\mathbb{F}_q$ so we must have 
\begin{equation}\label{equ 1} 
2d_1 \alpha + c_2 \alpha + d_1^2 + 2d_1 c_2  = 2d_2 \alpha + c_1 \alpha  + d_2^2 + 2d_2 c_1
\end{equation}
and 
\begin{equation}\label{equ 2}
2d_1 \beta + c_2 \beta + d_1^2 c_2 = 2d_2 \beta + c_1 \beta + d_2^2 c_1.
\end{equation}
Similar arguments show that both (\ref{equ 1}) and (\ref{equ 2}) hold with $c_3$ replacing $c_2$ and $d_3$ replacing $d_2$.  
We view $c_1$ and $d_1$ as begin fixed and $(c_2,d_2)$ and $(c_3,d_3)$ as solutions to the system 
\begin{equation}\label{equ 3} 
2d_1 \alpha + X \alpha +  d_1^2 + 2d_1 X  = 2 Y  \alpha + c_1 \alpha  + Y^2 + 2Y c_1,
\end{equation}
\begin{equation}\label{equ 4}
2d_1 \beta + X \beta + d_1^2 X = 2Y \beta + c_1 \beta + Y^2 c_1.
\end{equation}
One solution is $(X,Y) = (c_1 , d_1)$.  If we can show that the system (\ref{equ 3}), (\ref{equ 4}) has at most two solutions then we are done as this forces two of the pairs $(c_1,d_1), (c_2,d_2),(c_3,d_3)$ to be the same and 
the pair $(c_k ,d_k)$ uniquely determines the pair $(a_k , b_k)$.  Multiply 
(\ref{equ 3}) by $c_1$ and then subtract (\ref{equ 4}) to eliminate $Y^2$ and obtain 
\begin{equation}\label{equ 5}
( 2 c_1 d_1 \alpha + c_1 d_1^2 + c_1 \beta - 2d_1 \beta - c_1^2 \alpha ) + X( \alpha c_1 + 2c_1 d_1 - \beta 
- d_1^2 ) = Y ( 2c_1^2 + 2c_1 \alpha - 2 \beta ).
\end{equation}
Next we subtract $\alpha$ times (\ref{equ 4}) from $\beta$ times (\ref{equ 3}) to get 
\begin{equation}\label{equ 6}
d_1^2 \beta + X ( 2 d_1 \beta - d_1^2 \alpha )  = Y^2 ( \beta - \alpha c_1  ) + Y( 2c_1 \beta).
\end{equation}
If we knew that the coefficient of $X$ was nonzero in (\ref{equ 5}) and $\beta  - \alpha c_1 \neq 0$ then we could easily deduce that there are at most two solutions $(X,Y)$.  Unfortunately we do not know this and so we have to work to overcome this obstacle.  

Suppose (\ref{equ 5}) is an equation where the coefficients of $X$ and $Y$ are both 0.  Then 
\[
2c_1^2 + 2c_1 \alpha - 2 \beta = 0 ~~ \mbox{and} ~~ \alpha c_1 + 2c_1 d_1 - \beta - d_1^2 = 0.
\]
Since $q$ is odd, the first equation can be rewritten as $c_1^2 + c_1 \alpha - \beta$.  Subtracting the second equation 
$c_1^2 + c_1 \alpha - \beta$ gives $c_1^2 - 2c_1 d_1 + d_1^2 =0$ hence $(c_1 - d_1)(c_1 + d_1)=0$.  

If $c_1 = d_1$ then $\theta^{a_1} = \theta + c_1 = \theta + d_1 = \theta^{b_1}$ so $a_1 = b_1$ (in $\mathbb{Z}_{q^2 - 1}$).  Using $2b_1 - a_1 = 2b_2 - a_2$ we get $b_1 + a_2 = b_2 + b_2$ so $b_1 = b_2$, a contradiction.  Assume $c_1 = - d_1$.  Then $c_1 \neq 0$ and $d_1 \neq 0$ otherwise $c_1 = d_1$ which we already know does not occur.  Since both coefficients of $X$ and $Y$ are 0 in (\ref{equ 5}) the constant term must also be 0 so, using $c_1 = -d_1$, 
\begin{eqnarray*}
0 & = & 2c_1 d_1 \alpha + c_1 d_1^2 + c_1 \beta - 2d_1 \beta - c_1^2 \alpha \\
& = & -3c_1^2 \alpha + c_1^3 + 3c_1 \beta \\
& = & c_1 ( c_1^2 - 3c_1 \alpha + 3 \beta ).
\end{eqnarray*}
By (\ref{equ 0}) this is impossible.  We conclude that at least one of the coefficients of $X$ or $Y$ in (\ref{equ 5}) must be nonzero.  

If the coefficient of $X$ in (\ref{equ 5}) is nonzero then we can write $X = \gamma_1 Y + \gamma_2$ for some $\gamma_1 , \gamma_2 \in \mathbb{F}_q$.  Substituting this equation into (\ref{equ 3}) gives a quadratic equation in $Y$ which has at most two solutions and $Y$ uniquely determines $X$ since $X = \gamma_1 Y + \gamma_2$ and we are done.  

Assume now that $\alpha c_1 + 2 c_1 d_1 - \beta - d_1^2 = 0$.  Then (\ref{equ 5}) gives a unique solution for $Y$.  Since 
$(X,Y) = (c_1 , d_1)$ is a solution we must have that all solutions to the system (\ref{equ 3}), (\ref{equ 4}) have $Y = d_1$.  Substituting into (\ref{equ 3}) and (\ref{equ 4}) we get 
\begin{eqnarray*}
X ( \alpha + 2d_1 ) = c_1 ( \alpha + 2d_1) \\
X ( \beta + d_1^2 ) = c_1 ( \beta + d_1^2 ).
\end{eqnarray*}
If $d_1 = 0$ then $X \alpha = c_1 \alpha$ and since $\alpha \neq 0$ we get $X = c_1$ are we are done.  

Assume $d_1 \neq 0$.  If either 
$\alpha + 2d_1$ or $\beta +d_1^2$ are nonzero then we are done.  Assume 
$\alpha + 2d_1 = \beta + d_1^2 = 0$.  
If we substitute $Y = d_1$ into (\ref{equ 6}) then we get 
\[
X d_1 ( 2 \beta - d_1 \alpha ) = d_1 c_1 ( 2 \beta - d_1 \alpha).
\]
Again, if $2 \beta - d_1 \alpha$ is nonzero we are done so assume $2 \beta - d_1 \alpha = 0$.  Using the three equations 
\[
\alpha + 2d_1 = 0 , ~ \beta + d_1^2 = 0 ~ , 2 \beta  - d_1 \alpha = 0
\]
we have 
\[
0 = 2 \beta - d_1 \alpha = 2( - d_1^2 ) - d_1 ( -4d_1) = 2d_1^2
\]
so $d_1 = 0$ giving the needed contradiction.

\end{proof}


\end{document}